\documentclass[11pt]{amsart}
\usepackage[utf8]{inputenc}
\usepackage[english]{babel}
\usepackage[a4paper,top=4cm,bottom=4cm,left=4cm,right=4cm]{geometry}

\usepackage{doi} 
\usepackage{amstext}
\usepackage{latexsym}
\usepackage{amsmath}
\usepackage{amssymb}
\usepackage{amsfonts}
\usepackage{graphicx}
\usepackage{amssymb}
\usepackage{amsthm}
\usepackage{stackrel}
\usepackage{epsfig}
\usepackage{tikz}
\usetikzlibrary{matrix,arrows,decorations.pathmorphing}
\usepackage{graphicx}
\usepackage{enumitem}
\usepackage{xcolor}
\usepackage[misc]{ifsym}
\usepackage{hyperref}
\usepackage{comment}

\theoremstyle{definition}
\newtheorem{defn}{Definition}[section]
\newtheorem{rem}[defn]{Remark}

\theoremstyle{plain}

\newtheorem{lemma}[defn]{Lemma}
\newtheorem{prop}[defn]{Proposition}

\newtheorem{maintheorem}{Theorem}

\newcommand{\Id}{\operatorname{Id}}
\newcommand{\Lip}{\operatorname{Lip}}
\newcommand{\loc}{\operatorname{loc}}
\newcommand\Hol[1]{\mathcal{H}_{#1}}
\newcommand{\dsup}{d_{\infty}}

\newcommand\restr[2]{{ 
  \left.\kern-\nulldelimiterspace 
  #1 
  \vphantom{\big|} 
  \right|_{#2} 
  }}

\begin{document}

\title[Cohomology of cocycles of Lipschitz homeomorphisms]{Cohomology of cocycles of Lipschitz homeomorphisms}


\author{Marisa Cantarino}
\address{Instituto de Ciências Matemáticas e de Computação,Universidade de São Paulo, Brazil}
\email{mcantarino@icmc.usp.br}

\author{Catalina Freijo}
\address{Instituto de Matemática e Estatística, Universidade do Estado do Rio de Janeiro, Brazil\\}
\email{catalina.freijo@ime.uerj.br}

\thanks{M. C. was partially financed by the Australian Research Council (ARC) and by the BASAL fund FB210005 for center of excellence from ANID-Chile. C. F. was financed, in part, by the São Paulo Research Foundation (FAPESP), Brasil. We thank Aline Melo for several helpful discussions. We deeply appreciate the comments from the anonymous referees that greatly improved the exposition of the results in this work.}

\begin{abstract}
    We consider the set of H\"older continuous cocycles over a finite shift acting on the group of Lipschitz homeomorphisms $\Lip(M)$, where $M$ is a compact topological group. We establish that two dominated cocycles that coincide over periodic points of the base are cohomologous, with the conjugacy being H\"older continuous. Moreover, we prove that, under an additional condition, the existence of a measurable conjugacy implies the existence of a H\"older continuous conjugacy between the cocycles.
\end{abstract}

\maketitle	

\section{Introduction}

The notion of \textit{cocycle} in dynamical systems is defined by a map on a fiber bundle that acts on the fibers following a cocycle condition for compositions. More specifically, given a $\mathbb{Z}$-action $T: \mathbb{Z} \times X \to X$ and a topological group $H$, a \textit{cocycle} over $T$ with values in $H$ is given by $f: \mathbb{Z} \times X \to H$ satisfying 
$$f^{n+m}_x = f^m_{T^n(x)} f^n_x,$$
where $x \in X$, $n, m \in \mathbb{Z}$. We may consider $T$ the action of a more general group, but for most applications in dynamical systems $T$ is a homeomorphism on $X$. Fixing the dynamics $T: X \to X$, we say that the cocycle is simply $f: X \to H$.

A fundamental cocycle naturally arising in smooth dynamical systems is the Jacobian cocycle. This can be generalized to the broader framework of linear cocycles, where $f_x$ is a linear isomorphism on a vector space. Over the past decades, linear cocycles have been extensively studied, with applications to differential and measure-theoretic properties of both deterministic and random dynamical systems. Given their importance, a central problem is to classify such cocycles up to natural equivalence relations.

Two cocycles defined by $ f: X \to H $ and $ g: X \to H $ are called \textit{cohomologous} if there exists a \textit{transfer function} (or \textit{conjugacy}) $ \phi : X \to H $ such that  
$$  
f^n_x = \phi(T^n(x)) \, g^n_x \, \phi(x)^{-1}  
$$  
for all $ x \in X $ and $ n \in \mathbb{Z} $. The function $ \phi $ is typically required to ``preserve the structure'' of $ X $ and $ H $, meaning it is at least continuous or measurable, depending on the context.  

A special case arises when $ g^n_x \equiv \text{Id}_H $. If $ f $ is cohomologous to the trivial cocycle, then $ f $ is called a \textit{coboundary}.  

The problem of determining whether two cocycles are conjugate was initially studied in the context of coboundaries. Liv\v{s}ic proved in \cite{Liv71, Liv72} that, if $T: X \to X$ is hyperbolic and $H$ is a topological group admitting a bi-invariant complete metric, then $f$ is a coboundary if and only if $f^n_x = \Id_H$ for all $n$-periodic $x \in X$. This means that the property of being a coboundary is entirely characterized by the periodic points.

For the linear case, when $H = \text{GL}(m, \mathbb{R}) $, the work done in \cite{Kalinin2011} established a general result for characterizing when a cocycle is a coboundary. This work was later extended to diffeomorphism groups in low dimensions by \cite{KP16} and extended in \cite{AKL18} for higher dimensions. Previous developments addressed this question in various other settings, including connected Lie groups \cite{PollicottWalkden2001}, general Lie groups \cite{ddLW2010}, diffeomorphism groups \cite{NT95}, and Banach algebras \cite{BN98} --- though these cases all require technical conditions to control the behavior of the cocycle iterates.

It is worth mentioning some results in the literature that prove Liv\v{s}ic-like theorems with weaker hyperbolicity hypotheses for $T$, such as \cite{Wilk13, BP19, ZC2021}.

In the aforementioned works, the central question addressed is whether a given cocycle is cohomologous to the identity. To generalize this problem and study the conjugacy between two arbitrary cocycles, we introduce two key conditions. First, we say that $ f $ and $ g $ have \textit{coincidence of periodic data} if  
$$  
f^n_x = g^n_x \quad \text{for every } n\text{-periodic point } x \in X.  
$$  
This condition ensures that the two cocycles agree along periodic orbits, providing a necessary compatibility for their conjugacy. Second, the \textit{domination condition} (see Definition \ref{def:dom}) ensures that the \textit{skew-products} (see Definition \ref{def:skew-p}) associated with the cocycles considered in this work have invariant foliations by stable and unstable sets. Each one of these foliations, in turn, gives rise to a canonical mapping between transverse sections --- the \textit{holonomy} --- which plays a fundamental role in the construction of the conjugacy between $ f $ and $ g $.  

Remarkably, in many settings, these two conditions suffice to guarantee that $ f $ and $ g $ are cohomologous. Key results in this direction can be found in \cite{PollicottWalkden2001, Backes2015, Sadovskaya, BackesKocsard}, where this phenomenon is established across different dynamical contexts.  

\subsection{Results}

In this work, we consider a subshift $\Omega \subseteq \{1, \dots, k\}^{\mathbb{Z}}$
of finite type, the shift map $\sigma : \Omega \to \Omega$, and a cocycle
$f$ over $\sigma$ taking values in $(\mathrm{Lip}(M), \circ)$, where $M$ is a
compact topological group. Subshifts of finite type serve as canonical models for
uniformly hyperbolic dynamics and provide a well-understood symbolic base over which the complexity of the system is encoded in the fiber maps. The choice of
$\mathrm{Lip}(M)$ as the target group — the group of bi-Lipschitz homeomorphisms
of $M$ equipped with the Lipschitz metric $d_1$ — places the problem in a
genuinely non-smooth setting, extending the classical framework of linear and
differentiable cocycles to a broader class of fiber-wise transformations. Moreover, while the composition of two $\beta$-H\"older
maps is only $\beta^2$-H\"older, Lipschitz maps compose without loss of
regularity, making $\mathrm{Lip}(M)$ a natural and flexible target for a
cohomology theory in the non-differentiable regime. The relevant definitions are
recalled in Section~\ref{sec:preli} for the reader's convenience.

\begin{maintheorem}
    \label{teo:A}
    Consider $(M, d)$ a compact topological group with a bi-invariant metric $d$, and $f, g: (\Omega, d_{\Omega}) \to (\operatorname{Lip}(M), d_{\max,1})$ $\alpha$-H\"older cocycles over the transitive subshift of finite type $\sigma : \Omega \to \Omega$.

    If $f$ and $g$ are su-dominated and satisfy that, for every $n$-periodic point $x_0\in\Omega$,
        $$f^n_{x_0} = g^n_{x_0},$$
    then there are $\beta \in (0, 1)$ and an $\alpha \beta$-H\"older function $\phi: (\Omega, d_{\Omega}) \to (\Hol{\beta}(M), \dsup)$ such that
        $$f_{x}=\phi_{\sigma(x)} g_{x} \phi^{-1}_{x},$$
    for all $x \in \Omega$.
\end{maintheorem}

In the above result, $\Hol{\beta}(M)$ is the set of $\beta$-H\"older homeomorphisms of $M$ with $\beta$-H\"older inverse, and $d_{\max,1}$ is the metric $d_{\max,\beta}$ of Section~\ref{sec:preli} with $\beta = 1$. If $\gamma_{f}, \gamma_{g} < 1$ are respectively the H\"older exponents of the holonomies associated to $f$ and $g$ (see Proposition \ref{prop:stab-manifold-holon} and Remark \ref{rem:gamma}), then it is enough to consider $\beta \leq (\gamma_f \gamma_g)^3$ (see Lemma~\ref{lem:unif-fiber}). Transitivity is used to guarantee the density of the homoclinic classes employed in the proof.

We can consider a different kind of condition to obtain cohomologous 
cocycles with a H\"older continuous conjugacy by asking them to have 
a measurable conjugacy, which provides a rigidity result. Results in 
this direction go back to the work of Pollicott and 
Walkden~\cite{PollicottWalkden2001}, who established Liv\v{s}ic-type 
theorems for cocycles taking values in connected Lie groups, showing 
that measurable solutions to the cohomological equation are necessarily 
continuous. In the linear setting, Sadovskaya~\cite{Sadovskaya} proved 
that for fiber bunched $GL(d,\mathbb{R})$-valued cocycles, any measurable 
conjugacy must coincide with a H\"older continuous one, provided one of 
the cocycles is uniformly quasiconformal. This was further developed by 
Butler \cite{Butler2018}, who removed the fiber bunching condition on $A$ 
in the case $B \equiv Id_{\mathbb{R}^d}$, showing that measurable 
cohomology implies continuous cohomology under weaker assumptions.
In our setting, the role of the quasiconformality condition is played 
by the notion of bounded distortion: the required condition is 
that an arbitrary number of compositions of the cocycle generate bounded distortion at every point $x \in \Omega$ (see 
Definition~\ref{def:bounded-dist}).

    \begin{maintheorem}
    	\label{teo:B}
    	Let $\mu$ be an ergodic measure for the shift $\sigma: \Omega \to \Omega$  with full support and local product 
    	structure, and let $f,g \in \mathcal{H}_{\alpha}(\Omega, \Lip(M))$ be $\alpha$-H\"older cocycles over $\sigma$ having bounded distortion.
		Given $\beta \in (0,1]$, if $\phi: (\Omega, d_{\Omega})  \to (\mathcal{H}_{\beta}(M), \dsup)$ 
    	is a $\mu$-measurable conjugacy between $f$ and $g$ with a uniform bound for $H_{\beta}(\phi_x)$ and $H_{\beta}(\phi_x^{-1})$, then $\phi$ 
    	coincides $\mu$-almost everywhere with an $\alpha \beta$-H\"older 
    	continuous conjugacy 
    	$\tilde{\phi}: (\Omega, d_{\Omega}) \to (\Hol{\beta}(M), \dsup)$.
    \end{maintheorem}

The notion of bounded distortion for group-valued cocycles was introduced 
by Schmidt~\cite{S99}, motivated by earlier work of Nicol and 
Pollicott on compact Lie group extensions, where this condition holds 
automatically. In \cite{S99} the author proves a result similar to Theorem \ref{teo:B} for $H$ a Polish (separable and complete) group with a Lipschitz metric and requiring both $f$ and $g$ to have bounded distortion. His result does not apply directly to our case because $\Lip(M)$ with the Lipschitz distance is not separable (see Remark \ref{rem:lip-non-sep}).

The results in this paper hold for cocycles over any Lipschitz transitive hyperbolic homeomorphism on a compact metric space, but for simplicity we state and prove them for bilateral subshifts of finite type. Remember that transitive hyperbolic homeomorphisms satisfy a version of Smale's spectral decomposition by \cite{bowen70}, then they admit an extension by a subshift of finite type. The same spectral decomposition theorem also produces, within a transitive but non-mixing basic set, a cyclic decomposition into finitely many topologically mixing pieces. This is the smooth version of the period decomposition of an irreducible transition matrix that we use combinatorially in Section \ref{sec:A-transitive} (see for intance \cite[§4.5]{LindMarcus2021}). The use of the shift allows us to present the inequalities in the proofs in a cleaner way.

Both theorems may be interpreted as rigidity results: a necessary condition for cohomology (coincidence of periodic data, or existence of a measurable conjugacy) is shown to be sufficient under additional hypotheses. However, the resulting conjugacy is $\alpha\beta$-H\"older only with respect to $\dsup$, not with respect to the stronger metric $d_{\max,\beta}$. This loss of regularity is in analogy to the smooth case. Conditions or techniques to bootstrap the result to the stronger metric remain open.

\section{Preliminaries}
\label{sec:preli}

In this section, we introduce briefly some concepts, results and notations to be used in the proofs.

\subsection{Shift of finite type}

Consider the set $\{1,\ldots,k\}$ of $k$ \textit{symbols} and the 
bi-infinite sequences $\Sigma = \{1,\ldots,k\}^{\mathbb{Z}}$. The map 
$\sigma : \Sigma \to \Sigma$ defined by $\sigma(x_n)_{n\in\mathbb{Z}} 
= (x_{n+1})_{n\in\mathbb{Z}}$ is called the \textit{shift map}. Given a 
\textit{transition matrix} $\mathcal{A} = (\mathcal{A}_{i,j})_{1\leq i,j 
	\leq k}$, with $\mathcal{A}_{i,j} \in \{0,1\}$ and no null row, we define 
the \textit{subshift of finite type} associated to $\mathcal{A}$ as the 
$\sigma$-invariant subset
$$
\Omega = \{(x_n)_{n\in\mathbb{Z}} : \mathcal{A}_{x_n,x_{n+1}} = 1 
\; \text{for every} \; n \in \mathbb{Z}\} \subset \Sigma,
$$
with the restriction of $\sigma$ to it. Unlike full shifts, subshifts need 
not have fixed points, but those defined by transition matrices always 
have periodic points.

The transition matrix $\mathcal{A}$ is \emph{irreducible} if for every pair $i, j \in \{1, \ldots, k\}$ there is $n \geq 1$ such that $\mathcal{A}^n_{i,j} > 0$, and \emph{primitive} if there is $n \geq 1$ such that $\mathcal{A}^n_{i,j} > 0$ for every pair $i, j$. The subshift $\sigma : \Omega \to \Omega$ is topologically transitive if and only if $\mathcal{A}$ is irreducible, and topologically mixing if and only if $\mathcal{A}$ is primitive. An irreducible matrix has a \emph{period} $p \geq 1$, it is primitive if and only if $p = 1$, and the subshift admits a \emph{cyclic decomposition}: $\Omega$ is a disjoint union of closed sets $\Omega = \Omega_0 \sqcup \cdots \sqcup \Omega_{p-1}$ with $\sigma(\Omega_j) = \Omega_{j+1 \bmod p}$, such that the restriction of $\sigma^p$ to each $\Omega_j$ is topologically mixing (see, for instance, \cite[Section 4.5]{LindMarcus2021}). Notice also that if $\mathcal{A}$ is primitive, then $\mathcal{A}^n$ is primitive for every $n \geq 1$, so every power of a mixing subshift is mixing. From now on we assume that $\mathcal{A}$ is irreducible.

The space $\Omega$ carries a natural family of metrics: for $\rho > 1$, define
$$
d_{\rho}(x,y) = \rho^{-N_{x,y}},
$$
where $N_{x,y} = \max\{N \geq 0 : x_n = y_n \text{ for every } |n| < N\}$. 
Since the topologies induced by different values of $\rho$ are equivalent, 
we fix one such value and write $d_\Omega$ for the resulting metric.

The local and global hyperbolic structure of $\Omega$ is encoded in the 
stable and unstable sets of each point $x \in \Omega$. The 
\textit{global stable} and \textit{global unstable} sets are defined as
$$
W^s(x) := \{y \in \Omega : \exists\, n_0 \geq 0 
\text{ such that } x_n = y_n \mbox{ for all } n \geq n_0\},
$$
$$
W^u(x) := \{y \in \Omega : \exists\, n_0 \leq 0 
\text{ such that } x_n = y_n \mbox{ for all } n \leq n_0\},
$$
while their local counterparts are
$$
W_{\loc}^s(x) := \{y \in \Omega : x_n = y_n 
\mbox{ for all } n \geq 0\},
$$
$$
W_{\loc}^u(x) := \{y \in \Omega : x_n = y_n 
\mbox{ for all } n \leq 0\}.
$$
These sets admit a natural product description. For each symbol 
$i \in \{1,\ldots,k\}$, consider the cylinder $[0;i] = \{x \in \Omega : 
x_0 = i\}$ and the one-sided subshifts
$$
\Omega^{+} = \{(x_n)_{n\geq 0} : \mathcal{A}_{x_n,x_{n+1}} = 1 
\; \text{for every} \; n \geq 0\},
$$
$$
\Omega^{-} = \{(x_n)_{n\leq 0} : \mathcal{A}_{x_n,x_{n+1}} = 1 
\; \text{for every} \; n < 0\}.
$$
Let $\pi^s : \Omega \to \Omega^{+}$ and 
$\pi^u : \Omega \to \Omega^{-}$ denote the projections onto positive 
and negative coordinates respectively, and let 
$\psi_i : \pi^u([0;i]) \times \pi^s([0;i]) \to [0;i]$ be the 
natural inclusion. A $\sigma$-invariant measure $\mu$ is said to have 
\textit{local product structure} if it is locally equivalent to a product 
of its marginals in the following precise sense.

\begin{defn}
	Let $\mu$ be a $\sigma$-invariant measure and set $\mu^u = \pi^u_*\mu$, 
	$\mu^s = \pi^s_*\mu$. We say that $\mu$ has \emph{local product structure} 
	if there exists a continuous function $\Psi : \Omega \to (0,\infty)$ 
	such that for every $i \in \{1,\ldots,k\}$ and every measurable set 
	$E \subset [0;i]$,
	$$
	\mu(E) = \int (\chi_E \circ \psi_i)\, \Psi \; d\mu^u \times d\mu^s.
	$$
\end{defn}

\subsection{Cocycles with values in Lipschitz homeomorphisms}

Consider a compact topological group $(M, d)$. Throughout this work, we fix a compatible \emph{bi-invariant} metric $d$ on $M$, that is, a metric satisfying $d(ab, ac) = d(b, c)$ and $d(ba, ca) = d(b, c)$ for all $a, b, c \in M$. Such a metric always exists on a compact metrisable topological group: it is enough to average any compatible metric over the Haar measure in both variables. Notice that bi-invariance implies $d(a^{-1}, b^{-1}) = d(a, b)$ for all $a, b \in M$.

The set $\operatorname{Homeo}(M)$ 
denotes the collection of homeomorphisms of $M$. For $\beta \in (0,1]$, 
we say that $f \in \operatorname{Homeo}(M)$ is $\beta$-H\"older if there 
exists a constant $H_\beta(f) > 0$ such that
$$
d(f(p), f(q)) \leq H_{\beta}(f)\, d(p,q)^\beta \quad 
\text{for all } p, q \in M,
$$
where $H_\beta(f) = \sup_{p \neq q} \frac{d(f(p),f(q))}{d(p,q)^\beta}$ 
is the \textit{H\"older constant}, and $\beta$ is the \textit{H\"older exponent} of $f$. In the case that $\beta = 1$, we denote $L(f)=H_1(f)$ as the \textit{Lipschitz constant} of $f$. We define $\mathcal{H}_{\beta}(M) 
\subset \operatorname{Homeo}(M)$ as the subset of $\beta$-H\"older homeomorphisms whose 
inverses are also $\beta$-H\"older, equipped with the metric
$$
d_{\max,\beta}(f_1, f_2) = \max\{d_{\beta}(f_1, f_2), 
d_{\beta}(f_1^{-1}, f_2^{-1})\},
$$
where $d_{\beta}$ is the H\"older distance
$$
d_{\beta}(f_1, f_2) = \dsup(f_1, f_2) + 
H_{\beta}(f_1 \cdot f_2^{-1}),
$$
$f_1 \cdot f_2^{-1}$ denotes the map $p \mapsto f_1(p)\, f_2(p)^{-1}$ given by the group operation on $M$, and $\dsup(f_1, f_2) = \sup_{p \in M} d(f_1(p), f_2(p))$ is the 
uniform metric. When $\beta = 1$, the distance $d_1$ is called the 
\textit{Lipschitz metric} and $\mathcal{H}_1(M) = \Lip(M)$ is the space 
of bi-Lipschitz homeomorphisms of $M$.

\begin{rem}
	\label{rem:group-diff}
	The map $f_1 \cdot f_2^{-1}$ is $\beta$-H\"older whenever $f_1, f_2 \in \Hol{\beta}(M)$, so $d_\beta$ is well defined. Indeed, using the triangle inequality with the intermediate point $f_1(q) f_2(p)^{-1}$ and the bi-invariance of $d$,
	\begin{align*}
		d\big(f_1(p) f_2(p)^{-1}, f_1(q) f_2(q)^{-1}\big)
		&\leq d\big(f_1(p), f_1(q)\big) + d\big(f_2(p)^{-1}, f_2(q)^{-1}\big)\\
		&= d\big(f_1(p), f_1(q)\big) + d\big(f_2(p), f_2(q)\big)\\
		&\leq \big(H_\beta(f_1) + H_\beta(f_2)\big)\, d(p, q)^\beta.
	\end{align*}
\end{rem}
\begin{rem}
    \label{rem:lip-non-sep}
    The space of Lipschitz homeomorphisms with the metric $d_1$ is not separable. Indeed, take the uncountable family $\{h_b \}_{b \in (0, \frac{1}{2})} \subset \Lip(\mathbb{S}^1)$ of functions defined as
        $$h_b(p) =
        \begin{cases}
            \frac{3}{2}p & \text{if } 0 < p < b \\
            \frac{1}{2}p + b & \text{if } b < p < \frac{1}{2} \\
            \frac{3-4b}{2}p - \frac{1-4b}{2} & \text{if } \frac{1}{2} < p < 1
        \end{cases}.$$
    Then $L(h_b \cdot h_{b'}^{-1}) > \frac{1}{2}$ for all $b \neq b'$, and there could not be a countable dense set. Here the group operation on $\mathbb{S}^1 = \mathbb{R}/\mathbb{Z}$ is the addition modulo $1$, whose usual metric is bi-invariant, so $h_b \cdot h_{b'}^{-1}(p) = h_b(p) - h_{b'}(p) \bmod 1$ and the estimate follows from the jump of the slopes on the interval $(b, b')$.
\end{rem}

The following result is easily checked from the definitions.

\begin{lemma}
    \label{lem:d1}
    Consider $f_1, f_2$ and $f \in \Lip(M)$ and $L(f)$ the Lipschitz constant for $f$. We have that
    \begin{enumerate}
        \item $\dsup(f_1 \circ f, f_2 \circ f) = \dsup(f_1, f_2)$;
        \item $\dsup(f \circ f_1, f \circ f_2) \leq L(f) \dsup(f_1, f_2)$;
        \item $L(f_1 \circ f_2) \leq L(f_1) L(f_2)$.
    \end{enumerate}
\end{lemma}

Notice that item (3) above does not hold for H\"older functions, since the composition of two $\beta$-H\"older functions is only $\beta^2$-H\"older.

The following lemma guarantees that uniform limits of homeomorphisms with equibounded H\"older constants for the maps and their inverses stay in $\Hol\beta(M)$. It is used in the proofs of both Theorem~\ref{teo:A} and Theorem~\ref{teo:B}.

\begin{lemma}
	\label{lem:limit-homeo}
	Let $\psi_n \in \Hol{\beta}(M)$ be such that $\sup_n \max\{H_\beta(\psi_n), H_\beta(\psi_n^{-1})\} = C < \infty$ and $\psi_n \to \psi$ in $\dsup$. Then $\psi \in \Hol{\beta}(M)$, with $H_\beta(\psi) \leq C$ and $H_\beta(\psi^{-1}) \leq C$.
\end{lemma}

\begin{proof}
	The bound $H_\beta(\psi) \leq C$ follows by taking pointwise limits in the inequality $d(\psi_n(p), \psi_n(q)) \leq C d(p,q)^\beta$. For the inverse bound, notice that
	$$d(p, q) = d\big(\psi_n^{-1}\psi_n(p), \psi_n^{-1}\psi_n(q)\big) \leq C\, d\big(\psi_n(p), \psi_n(q)\big)^\beta,$$
	so $d(\psi_n(p), \psi_n(q)) \geq (d(p,q)/C)^{1/\beta}$, and passing to the limit the same holds for $\psi$. In particular $\psi$ is injective. It is also surjective: given $a \in M$, write $a = \psi_n(p_n)$ and take a convergent subsequence $p_n \to p$. Then $d(\psi(p), a) \leq d(\psi(p), \psi(p_n)) + d(\psi(p_n), \psi_n(p_n)) \to 0$, so $a = \psi(p)$. Hence $\psi$ is a homeomorphism of the compact space $M$ and, writing $a = \psi(p)$ and $b = \psi(q)$, the lower bound gives
	$$d(\psi^{-1}(a), \psi^{-1}(b)) = d(p, q) \leq C\, d(a, b)^\beta.$$
\end{proof}

We denote by $$\mathcal{H}_{\alpha}(\Omega, \Lip(M))=\{f: (\Omega, d_{\Omega})\to (\Lip(M), d_{\max,1}),\ \  \alpha\text{-H\"older}\}.$$ 

A \textit{cocycle} over $\sigma : \Omega \to \Omega$ is given by 
an element $f \in \mathcal{H}_{\alpha}(\Omega, \Lip(M))$, whose iterates 
are defined by
$$
f^n_x = f_{\sigma^{n-1}(x)} \circ \cdots \circ f_x, 
\quad \forall\, x \in \Omega.
$$

\begin{defn}
	\label{def:skew-p}
	The \textit{skew-product} associated to the cocycle 
	$f \in \mathcal{H}_{\alpha}(\Omega, \Lip(M))$ is the map 
	$F_f : \Omega \times M \to \Omega \times M$ defined by
	$$
	F_f(x, p) = (\sigma(x), f_x(p)).
	$$
\end{defn}

The iterates of $F_f$ satisfy $F_f^n(x,p) = (\sigma^n(x), f^n_x(p))$. 
We refer to $\sigma : \Omega \to \Omega$ as the \textit{base 
	dynamics} and to $f_x : M \to M$ as the \textit{fiber dynamics} 
on the \textit{fiber} $M$. When the cocycle $f$ is clear from context, 
we write $F$ instead of $F_f$.

\subsection{Holonomies}

For systems with some kind of \textit{partial hyperbolicity}, graph 
transform methods yield continuous families of foliations, as in 
\cite{HPS}. For cocycles, this partial hyperbolicity is expressed 
topologically through the following domination condition.

\begin{defn}
	\label{def:dom}
	We say that an $\alpha$-H\"older cocycle $f : \Omega \to \Lip(M)$ 
	is \emph{s-dominated} if there exists $\theta > 0$ such that
	\begin{equation}
		\label{eq:sdom}
		L(f_x^{-1}) < \rho^{\alpha - \theta} \quad \text{for all } x \in \Omega,
	\end{equation}
	where $\rho > 1$ is the constant used to define the metric $d_\Omega$ 
	on the shift $\sigma$. Analogously, $f$ is \emph{u-dominated} if there 
	exists $\theta > 0$ such that
	\begin{equation}
		\label{eq:udom}
		L(f_x) < \rho^{\alpha - \theta} \quad \text{for all } x \in \Omega.
	\end{equation}
	When $f$ satisfies both \eqref{eq:sdom} and \eqref{eq:udom}, we say 
	that it is \emph{su-dominated}.
\end{defn}

The s-domination condition implies that $L((f^n_x)^{-1}) < 
\rho^{n(\alpha-\theta)}$, meaning that the Lipschitz constant of 
$(f^n_x)^{-1}$ grows at a rate dominated by the hyperbolicity constant 
of the base. In particular, if $f$ is su-dominated, the skew-product $F$ 
acts on $\Omega \times M$ with stronger hyperbolicity in the base $\Omega$ 
than in the fiber $M$.

Given the skew-product $F : \Omega \times M \to \Omega \times M$, 
one can define maps between fibers $\{x\} \times M$ and $\{y\} \times M$ 
whenever $x$ and $y$ lie in the same stable set of the base dynamics. 
These maps are called holonomies and are made precise as follows.

\begin{defn}
    \label{def:holonomia}
    A family $h^s_{xy}: \{x\} \times M \to  \{y\} \times M$ with $y\in W^s(x)$ is an \emph{$s$-holonomy} for the skew-product $F$, if
    \begin{enumerate}
        \item $h^{s}_{yz} \circ h^{s}_{xy} = h^{s}_{xz}$ and $h^{s}_{xx} = id$,
        \item $f_y \circ h^{s}_{xy} = h^{s}_{\sigma(x)\sigma(y)} \circ f_x$,
        \item the map $(x, y, p) \mapsto h^{s}_{xy}(p)$ is continuous,
        \item the map $h^{s}_{xy}$ is H\"older continuous, with H\"older exponent independent of $x, y$, and with H\"older constant uniform over pairs $x, y$ with $y \in W^s_{\loc}(x)$.
    \end{enumerate}
\end{defn}

The definition for $u$-holonomies is analogous considering the inverse skew-product $F^{-1}(x, p) = (\sigma^{-1}(x), f_{\sigma^{-1}(x)}^{-1}(p))$.

The main reason to consider the property of su-domination in Definition \ref{def:dom} is that it allows us to prove the existence of holonomies. The following result is stated in \cite[Proposition 5.1]{AV2010} for $\alpha$-H\"older cocycles of diffeomorphisms. In our case the same proof applies and, in addition, we can get explicit constants from the fact that we work with symbolic dynamics in the base.

\begin{prop}
    \label{prop:stab-manifold-holon}
    Let $f \in \Hol\alpha( {\Omega},\Lip(M))$ be an s-dominated cocycle. Then there is $\gamma \in (0, 1)$ such that the associated skew-product has a family of $\gamma$-H\"older s-holonomies
        $$h^{s}_{xy}: \{x\} \times M \to  \{y\} \times M, \ \ y\in W^s(x),$$
	satisfying
        $$h^{s}_{xy}=\lim_{n\to\infty} (f_y^n)^{-1}f_x^n.$$
    Moreover, there is $C >0$ such that
    \begin{equation}
		\label{eq:hol-id}
        \dsup(h^{s}_{xy}, \operatorname{Id}) \leq C d_\Omega(x,y)^{\alpha},
    \end{equation}
    for all $x, y \in \Omega$ with $y \in W^s_{\loc}(x)$.
\end{prop}

\begin{proof}

Take first $y \in W^s_{\loc}(x)$. We show that the sequence of maps $\phi_{n} := (f_{y}^{n})^{-1} \circ f_{x}^{n}$
is Cauchy in $(C(M,M), d_{\infty})$.
Using the cocycle identity $f_{z}^{n+1} = f_{\sigma^{n} (z)} \circ f_{z}^{n}$,
for any $p \in M$ we set $q := f_{x}^{n}(p)$ and compute
\begin{align*}
    d\!\left(\phi_{n+1}(p),\, \phi_{n}(p)\right)
    &= d\!\left((f_{y}^{n})^{-1}
        \bigl((f_{\sigma^{n} (y)})^{-1} f_{\sigma^{n} (x)}(q)\bigr),\;
        (f_{y}^{n})^{-1}(q)\right) \\
    &\leq L\!\left((f_{y}^{n})^{-1}\right)
        d_{\infty}\!\left(f_{\sigma^{n}(x)}^{-1},\, f_{\sigma^{n}(y)}^{-1}\right).
\end{align*}
By the contraction
$d_{\Omega}(\sigma^{n} x, \sigma^{n} y) = \rho^{-n} d_{\Omega}(x,y)$
along local stable leaves,
$$\dsup(f_{\sigma^n(x)}^{-1},f_{\sigma^{n}(y)}^{-1})\leq H_{\alpha}(f) d_{\Omega}(\sigma^{n}(y),\sigma^{n}(x))^{\alpha} \leq H_{\alpha}(f) \rho^{-n\alpha}d_\Omega(x,y){^{\alpha}}.    
$$
Applying the domination condition~\eqref{eq:sdom} we have $L((f_{y}^{n})^{-1}) \leq \rho^{n(\alpha - \theta)}$, then:
$$
    d_{\infty}(\phi_{n+1}, \phi_{n})
    \leq H_{\alpha}(f)\, \rho^{n(\alpha-\theta)}\, \rho^{-n\alpha}\, d_{\Omega}(x,y)^{\alpha}
    = H_{\alpha}(f)\, \rho^{-n\theta}\, d_{\Omega}(x,y)^{\alpha}.
$$
Since $\theta > 0$ this is summable in $n$, so $(\phi_{n})$ is Cauchy and
converges uniformly to a continuous limit $h^{s}_{xy} := \lim_{n} \phi_{n}$.
Telescoping from $\phi_{0} = \mathrm{Id}$ and summing the geometric series:
$$
    d_{\infty}(h^{s}_{xy},\, \mathrm{Id})
    \leq \sum_{n=0}^\infty H_{\alpha}(f)\, \rho^{-n\theta}\, d_{\Omega}(x,y)^{\alpha}
    = \frac{H_{\alpha}(f)}{1 - \rho^{-\theta}}\, d_{\Omega}(x,y)^{\alpha},
$$
which gives  inequality \ref{eq:hol-id} in the statement with
$C = H_{\alpha}(f) / (1 - \rho^{-\theta})$.

Since $(\Lip(M), \dsup)$ is not complete, the holonomy is just a continuous function a priori. A similar argument as above can be applied to $(f_{x}^{n})^{-1} \circ f_{y}^{n}$, also using s-domination, and we have that $h^{s}_{xy}$ is a homeomorphism. It remains to prove that $h^{s}_{xy}$ are s-holonomies.

The composition law and the identity at $y = x$ at (1) in Definition \ref{def:holonomia} follow immediately
from the definition $h^{s}_{xy} = \lim_{n} (f_{y}^{n})^{-1} f_{x}^{n}$, since
$(f_{z}^{n})^{-1} f_{y}^{n} \circ (f_{y}^{n})^{-1} f_{x}^{n} = (f_{z}^{n})^{-1} f_{x}^{n}$
and $(f_{x}^{n})^{-1} f_{x}^{n} = \mathrm{id}$ for every $n$.
The equivariance condition (2),
$f_{y} \circ h^{s}_{xy} = h^{s}_{\sigma(x), \sigma(y)} \circ f_{x}$,
follows by passing to the limit in the identity
$f_{y} \circ (f_{y}^{n})^{-1} f_{x}^{n} = (f_{\sigma(y)}^{n-1})^{-1} f_{\sigma(x)}^{n-1} \circ f_{x}$,
which holds for every $n$. Continuity (3) is a consequence of the uniform convergence of the continuous maps $(x, y, p) \mapsto \phi_n(p)$, with the Cauchy estimate above being uniform for $y \in W^s_{\loc}(x)$. On global stable sets, continuity follows from the equivariance relation, as at the end of this proof.

Now, to prove that the H\"older constants are uniform on local leaves ((4) in Definition \ref{def:holonomia}), fix $x, y \in \Omega$ with $y \in W^{s}_{\loc}(x)$, and fix $p, q \in M$.
Let $L_{0} := \sup_{z \in \Omega} L(f_{z}) < \infty$.
By the triangle inequality:
$$
    d(h^{s}_{xy}(p),\, h^{s}_{xy}(q))
    \leq 2 \underbrace{d_{\infty}(h^{s}_{xy}, \phi_{n})}_{\text{(i)}}
    + \underbrace{d(\phi_{n}(p),\, \phi_{n}(q))}_{\text{(ii)}}.
$$
For (i), summing the Cauchy bound starting at index $n$:
$$
    d_{\infty}(h^{s}_{xy},\, \phi_{n})
    \leq \frac{H_{\alpha}(f)\, \rho^{-n\theta}}{1 - \rho^{-\theta}}\, d_{\Omega}(x,y)^{\alpha} \leq \frac{H_{\alpha}(f)\, \rho^{-n\theta}}{1 - \rho^{-\theta}}.
$$
For (ii), the Lipschitz bounds on $(f_{y}^{n})^{-1}$ and $f_{x}^{n}$ give
\begin{align*}
	d(\phi_{n}(p),\, \phi_{n}(q)) &\leq L\!\left((f_{y}^{n})^{-1}\right) \cdot L(f_{x}^{n}) \cdot d(p,q)\\
    &\leq \rho^{n(\alpha-\theta)}\, L_{0}^{n}\, d(p,q) = \rho^{-n\theta}\, K^{n}\, d(p,q),
\end{align*}
where $K = \rho^{\alpha} L_{0} > 1$.

Choose $n$ such that $K^{-n} \leq d(p,q)$ and hence $K^{n} d(p,q) \leq 1 \leq K$. Then (ii) is bounded by $K \rho^{-n\theta}$.
Since $n \geq \frac{\log d(p,q)^{-1}}{\log K} - 1$, we have

    $$\rho^{-n\theta}
    \leq \rho^{\theta} \cdot \rho^{-\frac{\theta}{\log K}
        \log d(p,q)^{-1}}
    = \rho^{\theta} \cdot d(p,q)^{\gamma},$$
where
    $$\gamma := \frac{\theta \log \rho}{\log K}
            = \frac{\theta \log \rho}{\alpha \log \rho + \log L_{0}} > 0.$$
Assembling all three terms:
$$
    d(h^{s}_{xy}(p),\, h^{s}_{xy}(q))
    \leq \rho^{\theta}\left(K +
        \frac{2H_{\alpha}(f)}{1 - \rho^{-\theta}}\right)
    d(p,q)^{\gamma}.
$$
Thus the H\"older exponent $\gamma > 0$ and the H\"older constant are both uniform for pairs $x, y$ with $y \in W^s_{\loc}(x)$.

For a general pair $y \in W^s(x)$, take $n \geq 0$ such that $\sigma^n(y) \in W^s_{\loc}(\sigma^n(x))$. The equivariance condition (2) gives
$$h^s_{xy} = (f^n_y)^{-1} \circ h^s_{\sigma^n(x)\sigma^n(y)} \circ f^n_x,$$
so $h^s_{xy}$ is $\gamma$-H\"older with H\"older constant bounded by $L((f^n_y)^{-1})\, L(f^n_x)^\gamma$ times the uniform local constant. In particular, the exponent $\gamma$ does not depend on $x, y$, while the H\"older constant may deteriorate along global stable sets with the growth of the Lipschitz constants of the iterates of the cocycle.

\end{proof}

\begin{rem}
	\label{rem:gamma}
	The H\"older exponent $\gamma$ of the holonomy depends on the constant $\theta$ from the definition of domination and on $L_0$.
	The lowest value $L_0$ can assume is $1$, since $f_x$ are homeomorphisms on a compact space. That happens in the case that each $f_x$ is an isometry. Then $\gamma = \theta/\alpha$. But, again, because $f_x$, $f_x^{-1}$ are homeomorphisms on a compact space, $\theta > 0$ in the definition of s-domination must always be smaller than $\alpha$, then in the case of a cocycle of isometries, any $\theta \in (0, \alpha)$ satisfies the definition of domination, and we can have $\gamma$ as close to $1$ as we want. This is consistent with the fact that the holonomies are also isometries. In general, as $L_0$ increases or $\theta$ decreases, we get less regular holonomies.
\end{rem}

\subsection{Bounded distortion}

To prove Theorem \ref{teo:B}, we need a property that is stronger than su-domination. 

\begin{defn}
    \label{def:bounded-dist}
    We say that a cocycle $f: \Omega \to \Lip(M)$ has \emph{bounded distortion} if there is $K_f \geq 0$ such that, for all $n \in \mathbb{N}$ and all $x \in \Omega$, we have
        \begin{equation}
            \label{eq:top-qc}
            \begin{aligned}
                L(f_x^n) < K_f \mbox{ and }\\
                L((f_x^n)^{-1}) < K_f.                
            \end{aligned}
        \end{equation}
\end{defn}

It is straightforward to see that bounded distortion implies su-do\-mi\-na\-tion. If we consider $\operatorname{Isom}(M) \subseteq \Lip(M)$ the subgroup of isometries, the bounded distortion is trivially satisfied. If $f$ is a projective cocycle acting on $\mathbb{RP}^1 \sim \mathbb{S}^1$ given by the quotient of a quasiconformal linear cocycle in $\operatorname{Gl}(2, \mathbb{R})$, then each $f_x$ is an isometry \cite{Kalinin-Sadovskaya}. A non-isometric example is given by a cocycle $f: \{0, 1\}^{\mathbb{Z}} \to \Lip(\mathbb{S}^1)$ with 
$$f_x = f_{x_0} = \begin{cases}
f_0, &\mbox{ if } x_0 = 0\\
f_1, &\mbox{ if } x_0 = 1
\end{cases},$$
where $f_0$ and $f_1$ are circle homeomorphisms with irrational rotation numbers and are conjugate to the irrational rotations $R_{\omega_1}$ and $R_{\omega_2}$ by the same bi-Lipschitz $h: \mathbb{S}^1 \to \mathbb{S}^1$. Then arbitrary compositions of $f_0$ and $f_1$ would be conjugated to some irrational rotation with the same $h$, thus the compositions have Lipschitz constants bounded by $L(h)^2$. The same applies for their inverses and $L(h^{-1})^2$.

\section{Proof of Theorem \ref{teo:A}}

The theorem is proved through constructing a conjugacy map. Initially defined on a dense subset of the base, we establish its $\alpha\beta$-H\"older continuity within this set. Consequently, this property extends to the closure, ensuring $\alpha\beta$-H\"older continuity across the entire domain.

We first prove the theorem when the transition matrix $\mathcal{A}$ is primitive, that is, when $\sigma$ is topologically mixing (Sections \ref{sec:A-fixed} and \ref{sec:A-per}). The general transitive case is reduced to the mixing case in Section \ref{sec:A-transitive}.

\subsection{Case when the shift has a fixed point}
\label{sec:A-fixed}

To construct the conjugacy, we assume first that $x_0$ is a fixed point for $\sigma$. We define the conjugacy $\phi$ within the homoclinic class of $x_0$, which is $W^s(x_0)\cap W^u(x_0)$. Since $\sigma$ is topologically mixing, this class is dense in $\Omega$. Thus, let $\phi : W^s(x_0)\cap W^u( x_0)\to \Hol{\beta}(M)$ be defined by 
\begin{equation}
    \label{eq:def-phi}
    \phi_y = h^{s,f}_{ x_0y}(h^{s,g}_{ x_0y})^{-1}.
\end{equation}
Each $\phi_y$ is the composition of a $\gamma_f$-H\"older and a $\gamma_g$-H\"older homeomorphism, where $\gamma_f$ and $\gamma_g$ are the H\"older exponents of the holonomies for the skew-products $F$ and $G$ respectively (see Proposition \ref{prop:stab-manifold-holon}), so $\phi_y \in \Hol{\gamma_f\gamma_g}(M)$ for each $y$. Since the holonomies in \eqref{eq:def-phi} are taken along global stable sets, their H\"older constants, and hence those of $\phi_y$, are not uniform in $y$ a priori. Uniform bounds are obtained in Lemma \ref{lem:unif-fiber} below, at the price of a smaller H\"older exponent. The function $\phi$ satisfies $f_{y} = \phi_{\sigma(y)}g_{y} \phi^{-1}_{y},$ for every $y \in W^s(x_0)\cap W^u(x_0)$.
Indeed, from the holonomies properties we have
\begin{align*} 
    (h^{s,f}_{ x_0\sigma(y)})^{-1}f_yh^{s,f}_{ x_0y} &= f_{ x_0}\\
    (h^{s,g}_{ x_0\sigma(y)})^{-1}g_yh^{s,g}_{ x_0y} &= g_ {x_0}=f_ {x_0}.\\
\end{align*}

By joining both equations, we have that
    $$(h^{s,f}_{ x_0\sigma(y)})^{-1}f_yh^{s,f}_{ x_0y} = (h^{s,g}_{ x_0\sigma(y)})^{-1}g_yh^{s,g}_{ x_0y},$$
thus
$$ f_y = h^{s,f}_{ x_0\sigma(y)}(h^{s,g}_{ x_0\sigma(y)})^{-1} g_y h^{s,g}_{ x_0y}(h^{s,f}_{ x_0y})^{-1},$$
which implies
$$f_y = \phi_{\sigma(y)} g_y \phi_y^{-1}.$$

Before proving that the function $\phi$ defined above is $\alpha \beta$-H\"older with respect to the uniform metric, we first prove that we can also use  unstable holonomies to define the same $\phi$.

\begin{lemma}
	\label{lema1}
	Given $x_0$ a fixed point for $\sigma$, we have that
	$$\phi_y=h^{s,f}_{x_0y}h^{s,g}_{yx_0}=h^{u,f}_{x_0y}h^{u,g}_{yx_0}$$ for every $y\in W^s(x_0)\cap W^u(x_0)$.
\end{lemma}

\begin{proof}
	We want to verify that both of the following sequences converge to the same limit with respect to the metric $\dsup$
	\begin{align*}
		h^{s,f}_{x_0y}h^{s,g}_{yx_0} = &\lim_{n\to\infty}(f^n_y)^{-1}f^n_{x_0}(g^n_{x_0})^{-1}g^n_y=\lim_{n\to\infty}(f^n_y)^{-1}g^n_y,\\
		h^{u,f}_{x_0y}h^{u,g}_{yx_0}=&\lim_{n\to\infty}(f^{-n}_y)^{-1}f^{-n}_{x_0}(g^{-n}_{x_0})^{-1}g^{-n}_y=\lim_{n\to\infty}(f^{-n}_y)^{-1}g^{-n}_y.
	\end{align*}
	Thus we want to verify that
	\begin{equation}
		\label{eq:2}
		\lim_{n\to\infty} \dsup((f^{-n}_y)^{-1}g^{-n}_y,(f^n_y)^{-1}g^n_y) = 0.
	\end{equation}
	
	Given $y\in W^s(x_0)\cap W^u(x_0)$, we know that there exists $n_0 > 0$ such that $\sigma^{n_0}(y) \in W_{\loc}^s(x_0)$ and $\sigma^{-n_0}(y) \in W_{\loc}^u(x_0)$. This implies that, for $n > n_0$, 
	$$d_{\Omega}(\sigma^{n}(y),x_0) < \rho^{-(n-n_0)}\text{ and }
	d_{\Omega}(\sigma^{-n}(y),x_0) < \rho^{-(n-n_0)}.$$
	Combining these two inequalities we have that
	\begin{equation}
		\label{eq:n-n}
		d_\Omega(\sigma^{n}(y),\sigma^{-n}(y)) < 2 \rho^{-(n-n_0)}.
	\end{equation}
	
	Given $\epsilon > 0$, there exists $n_1 > n_0 > 0$ such that for every $n > n_1$ the set $\Xi_n = \{\sigma^{-n}(y),\ldots, y,\ldots \sigma^{n}(y)\}$ is an $\epsilon$-pseudo orbit. Then, the Anosov Closing Lemma implies that there are a $\delta > 0$ and a $2n$-periodic point $z_n$ such that $z_n$ $\delta$-shadows the pseudo orbit $\Xi_n$. Additionally, along this shadowing orbit we have some exponential closeness of the distances to the pseudo orbit. More precisely, for all $0 \leq j \leq 2n$, we have
	\begin{equation}
		\label{eq:closing}
		d_\Omega(\sigma^{j}(\sigma^{-n}(y)), \sigma^{j}(\sigma^{-n}(z_n))) < \rho^{-\min\{j, 2n-j\}} d_\Omega(\sigma^{n}(y),\sigma^{-n}(y)).
	\end{equation}
	
	We can bound the sequence in equation (\ref{eq:2}) with 
	\begin{equation}
		\label{eq:4}
		d_{ \infty}((f^{-n}_y)^{-1}g^{-n}_y,(f^{-n}_{z_n})^{-1}g^{-n}_{z_n})) +d_{ \infty}((f^n_{z_n})^{-1}g^n_{z_n},(f^n_y)^{-1}g^n_y)),
	\end{equation}
	since
	$$(f^{-n}_{z_n})^{-1}g^{-n}_{z_n}=(f^n_{z_n})^{-1}g^n_{z_n},$$
	which holds as a consequence of    
	$$g^n_{z_n} (g^{-n}_{z_n})^{-1} = g^n_{z_n} g^n_{\sigma^{-n}(z_n)} = g^{2n}_{\sigma^{-n}(z_n)}=f^{2n}_{\sigma^{-n}(z_n)} = f^n_{z_n} (f^{-n}_{z_n})^{-1}.$$
	
	To prove the convergence of equation (\ref{eq:4}) to zero, we establish an upper bound for the second term, $\dsup((f^n_{z_n})^{-1}g^n_{z_n},(f^n_y)^{-1}g^n_y)$, with the first term being analogous.
	
	The su-domination implies that
	\begin{equation}
		\label{eq:A1}
		\begin{aligned}
			\dsup((f^n_{z_n})^{-1}g^n_{z_n},(f^n_y)^{-1}g^n_y)
			&\leq L((f^n_{z_n})^{-1}) \dsup(g^n_{z_n} (g^n_y)^{-1},f^n_{z_n}(f^n_y)^{-1})\\
			&\leq \rho^{n(\alpha-\theta)} \left( \dsup(g^n_{z_n} (g^n_y)^{-1}, \Id) + \dsup(\Id, f^n_{z_n}(f^n_y)^{-1}) \right).
		\end{aligned}
	\end{equation}
	
	Again, here we bound $\dsup(f^n_{z_n}(f^n_y)^{-1}, \Id)$, with $\dsup(g^n_{z_n} (g^n_y)^{-1}, \Id)$ being analogous. Using a telescopic sum, the triangle inequality and the su-domination, we have that
	$$\dsup(f^n_{z_n}(f^n_y)^{-1}, \Id) \leq \sum_{j=0}^{n-1} \dsup \left(f^{n-j}_{\sigma^j(z_n)}(f^{n-j}_{\sigma^j(y)})^{-1}, f^{n-j-1}_{\sigma^{j+1}(z_n)}(f^{n-j-1}_{\sigma^{j+1}(y)})^{-1} \right).$$
	This expression can be rewritten as 
	$$\sum_{j=0}^{n-1}\dsup \left(f^{n-j-1}_{\sigma^{j+1}(z_n)}f_{\sigma^{j}(z_n)}(f_{\sigma^{j}(y)})^{-1}(f^{n-j-1}_{\sigma^{j+1}(y)})^{-1}, f^{n-j-1}_{\sigma^{j+1}(z_n)}(f^{n-j-1}_{\sigma^{j+1}(y)})^{-1} \right),$$
	which is smaller than
	\begin{align*}\sum_{j=0}^{n-1}L(f^{n-j-1}_{\sigma^{j+1}(z_n)})& \dsup \left(f_{\sigma^{j}(z_n)}(f_{\sigma^{j}(y)})^{-1}, \Id \right)\\ &\leq
		\sum_{j=0}^{n-1} \rho^{(n-j-1)(\alpha-\theta)} \dsup \left(f_{\sigma^{j}(z_n)}(f_{\sigma^{j}(y)})^{-1}, \Id \right).\end{align*}
	For $0 \leq j <n$, we have that
	$$\dsup \left(f_{\sigma^{j}(z_n)}(f_{\sigma^{j}(y)})^{-1}, \Id \right) = \dsup(f_{\sigma^{j}(z_n)}, f_{\sigma^{j}(y)}),$$
	is bounded by     
	$$
	H_{\alpha}(f) d_\Omega(\sigma^{j}(z_n), \sigma^{j}(y))^{\alpha} = H_{\alpha}(f) d_\Omega(\sigma^{n+j}(\sigma^{-n}(z_n)), \sigma^{n+j}(\sigma^{-n}(y)))^{\alpha},
	$$ where $H_{\alpha}(f)$ is the H\"older constant of $f: \Omega \to \Lip(M)$. Using equations (\ref{eq:n-n}) and (\ref{eq:closing}), this becomes smaller than
	$$
		H_{\alpha}(f) \rho^{-\alpha(n-j)} d_\Omega(\sigma^{n}(y),\sigma^{-n}(y)) \leq 2^{\alpha} H_{\alpha}(f) \rho^{-\alpha(n-j)} \rho^{-\alpha(n-n_0)}.  
	$$
	Therefore    
	\begin{align*}
		\dsup(f^n_{z_n}(f^n_y)^{-1}, \Id) &\leq 2^{\alpha} H_{\alpha}(f) \rho^{-\alpha(n-n_0)} \sum_{j=0}^{n-1} \rho^{(n-j-1)(\alpha-\theta)} \rho^{-\alpha(n-j)}\\
		&= 2^{\alpha} H_{\alpha}(f) S \rho^{-\alpha(n-n_0)},
	\end{align*}
	where $S = \sum_{j=0}^{n-1} \rho^{(n-j-1)(\alpha-\theta)} \rho^{-\alpha(n-j)}$.
	
	Going back to equation (\ref{eq:A1}), this implies that
	\begin{align*}
		\dsup((f^n_{z_n})^{-1}g^n_{z_n},(f^n_y)^{-1}g^n_y) & \leq 2^{\alpha+1} H_{\alpha}(f) S \rho^{n(\alpha-\theta)} \rho^{-\alpha(n-n_0)} \\ &= 2^{\alpha+1} H_{\alpha}(f) S \rho^{-n\theta + \alpha n_0},
	\end{align*}
	which converges to $0$ as $n$ increases.
\end{proof}

We now prove that the fibre maps $\phi_y$ have uniformly bounded H\"older constants over the homoclinic class. We denote $\mathcal{C} = W^s(x_0)\cap W^u(x_0)$.

\begin{lemma}
	\label{lem:unif-fiber}
	There exist $\beta \in (0,1)$ and $C_\phi > 0$ such that, for every $y \in \mathcal{C}$, we have $\phi_y \in \Hol{\beta}(M)$ with
	$$H_\beta(\phi_y) \leq C_\phi \quad \text{and} \quad H_\beta(\phi_y^{-1}) \leq C_\phi.$$
	It is enough to take any $\beta \leq (\gamma_f\gamma_g)^3$.
\end{lemma}

\begin{proof}
	Let $z, w \in \mathcal{C}$ with $w \in W^s(z)$. Since $x_0$, $w$ and $z$ belong to the same global stable set, the composition law (1) in Definition \ref{def:holonomia} gives $h^{s,f}_{x_0z} = h^{s,f}_{wz}\, h^{s,f}_{x_0w}$ and $h^{s,g}_{zx_0} = h^{s,g}_{wx_0}\, h^{s,g}_{zw}$, hence, by the definition of $\phi$ in equation \ref{eq:def-phi} and by Lemma \ref{lema1},
	\begin{equation}
		\label{eq:chain-s}
		\phi_z = h^{s,f}_{wz}\, \phi_w\, h^{s,g}_{zw}.
	\end{equation}
	Analogously, using the unstable representation of $\phi$ in Lemma \ref{lema1}, if $y, w \in \mathcal{C}$ with $w \in W^u(y)$, then
	\begin{equation}
		\label{eq:chain-u}
		\phi_w = h^{u,f}_{yw}\, \phi_y\, h^{u,g}_{wy}.
	\end{equation}

	Now fix $0 < \varepsilon < 1$, so that the bracket $[z, y] = W^s_{\loc}(z) \cap W^u_{\loc}(y)$ consists of a single point whenever $d_\Omega(z, y) < \varepsilon$. Since $\mathcal{C}$ is dense and $\Omega$ is compact, there is a finite $\varepsilon$-dense set $F \subset \mathcal{C}$. Define
	$$C_F := \max_{y \in F}\, \max\{H_{\gamma_f\gamma_g}(\phi_y),\ H_{\gamma_f\gamma_g}(\phi_y^{-1})\} < \infty.$$
	For every $\varepsilon$, since $F$ is finite, $C_F$ is also finite, by Proposition \ref{prop:stab-manifold-holon} and the remark on H\"older constants for global stable sets at the end of its proof. No uniformity over $\mathcal{C}$ is claimed at this point.

	Given $z \in \mathcal{C}$, take $y \in F$ with $d_\Omega(z, y) < \varepsilon$ and set $w = [z, y]$. Then $w \in \mathcal{C}$, since $w \in W^s_{\loc}(z) \subset W^s(x_0)$ and $w \in W^u_{\loc}(y) \subset W^u(x_0)$. Combining \eqref{eq:chain-s} and \eqref{eq:chain-u},
	$$\phi_z = h^{s,f}_{wz}\, h^{u,f}_{yw}\, \phi_y\, h^{u,g}_{wy}\, h^{s,g}_{zw}.$$
	All four holonomies above are taken along local stable or local unstable sets, so their H\"older constants are uniformly bounded by Proposition \ref{prop:stab-manifold-holon}. Then $\phi_z$ is $\beta$-H\"older for $\beta = (\gamma_f\gamma_g)^3$, with H\"older constant bounded by an explicit function of $C_F$ and the uniform local holonomy constants, hence independent of $z$. The same argument applies to
	$$\phi_z^{-1} = h^{s,g}_{wz}\, h^{u,g}_{yw}\, \phi_y^{-1}\, h^{u,f}_{wy}\, h^{s,f}_{zw},$$
	since $(h_{xy})^{-1} = h_{yx}$.
	
	Since $M$ is compact, the statement follows for any $\beta \leq (\gamma_f\gamma_g)^3$.
\end{proof}

\begin{rem}
	\label{rem:bd-alternative}
	If both $f$ and $g$ have bounded distortion (Definition \ref{def:bounded-dist}), the loss of exponent above can be avoided. The argument at the end of the proof of Proposition \ref{prop:stab-manifold-holon} bounds the global H\"older constant of $h^{s,f}_{xy}$ by $L((f^n_y)^{-1})\, L(f^n_x)^\gamma$ times the local constant, with $n=n(x,y)$. Bounded distortion makes $L((f^n_y)^{-1})$ and $L(f^n_x)$ uniformly bounded in $n$, hence this global constant is already uniform in $x,y$, with no need to pass to local leaves. In that case $\phi_y$ and $\phi_y^{-1}$ are already uniformly $\gamma_f\gamma_g$-H\"older on $\mathcal{C}$, Lemma \ref{lem:unif-fiber} becomes unnecessary, and one obtains $\beta = \gamma_f\gamma_g$ in Theorem \ref{teo:A} instead of $(\gamma_f\gamma_g)^3$. We keep Theorem \ref{teo:A} under the weaker su-domination hypothesis.
\end{rem}

We now show that the function $\phi$ defined above is $\alpha\beta$-H\"older 
with respect to the uniform metric $\dsup$, with $\beta$ as in Lemma \ref{lem:unif-fiber}. Since 
the uniform limit of H\"older continuous functions with bounded H\"older 
constant is again H\"older continuous, this regularity extends to the 
closure, and Lemmas \ref{lem:unif-fiber} and \ref{lem:limit-homeo} guarantee that the limit fibre maps belong to $\Hol\beta(M)$: the function $\phi$ admits an $\alpha\beta$-H\"older extension 
$$\phi : \Omega = \overline{W^s(x_0)\cap W^u(x_0)} \to \mathcal{H}_{\beta}(M).$$

\begin{lemma}
	\label{lem:phi-lips}
	There exists $C > 0$ such that 
	$$\dsup(\phi_y, \phi_z) \leq C\, d_\Omega(y,z)^{\alpha \beta} \text{ for all }y, z \in W^s(x_0)\cap W^u(x_0),$$ where $\beta$ is the exponent given by Lemma \ref{lem:unif-fiber}.
\end{lemma}
\begin{proof}
	Since $\Omega$ is compact, we can find the H\"older constant $C>0$ 
	locally for finitely many balls and take the maximum. Thus, it is 
	enough to prove the H\"older inequality for $y, z \in \mathcal{C}$ with $d_\Omega(y, z) < \varepsilon$, with $\varepsilon < 1$ as in the proof of Lemma \ref{lem:unif-fiber}.
	
	Consider $w = W^s_{\loc}(y)\cap W^u_{\loc}(z)$, 
	the local product of $y$ and $z$, and notice that $w \in \mathcal{C}$. By the triangle inequality,
	$$\dsup(\phi_y, \phi_z) \leq \dsup(\phi_y, \phi_w) + \dsup(\phi_w, \phi_z).$$
	We prove the H\"older inequality for $\dsup(\phi_y, \phi_w)$.
	Since $w \in W^s_{\loc}(y)$, equation \eqref{eq:chain-s} gives $\phi_w = h^{s,f}_{yw}\, \phi_y\, h^{s,g}_{wy}$, hence
	\begin{align*}
		\dsup(\phi_y, \phi_w) &\leq \dsup(\phi_y,\, \phi_y\, h^{s,g}_{wy}) + \dsup(\phi_y\, h^{s,g}_{wy},\, h^{s,f}_{yw}\, \phi_y\, h^{s,g}_{wy})\\
		&\leq H_\beta(\phi_y)\, \dsup(\Id,\, h^{s,g}_{wy})^\beta + \dsup(\Id,\, h^{s,f}_{yw}),
	\end{align*}
	where the second term uses item (1) of Lemma \ref{lem:d1}. By Lemma \ref{lem:unif-fiber}, $H_\beta(\phi_y) \leq C_\phi$, and, since $(y, w)$ is a local pair, Proposition~\ref{prop:stab-manifold-holon} gives
	$$\dsup(h^{s,f}_{yw}, \operatorname{Id}) \leq C_1\, d_\Omega(y, w)^{\alpha} 
	\quad \text{and} \quad \dsup(h^{s,g}_{wy}, \operatorname{Id}) 
	\leq C_1\, d_\Omega(y, w)^{\alpha}.$$
	This implies
	\begin{align*}
		\dsup(\phi_y, \phi_w) &\leq C_\phi C_1^\beta\, d_\Omega(y, w)^{\alpha\beta} + C_1\, d_\Omega(y, w)^{\alpha} \leq C_2\, 
		d_\Omega(y, w)^{\alpha\beta}.
	\end{align*}
	The term $\dsup(\phi_w, \phi_z)$ is analogous. This gives
	\begin{equation}
		\label{eq:bound-d}
		\dsup(\phi_y, \phi_z) \leq C_2\left(d_\Omega(y, w)^{\alpha\beta} + 
		d_\Omega(w, z)^{\alpha\beta}\right),
	\end{equation}
	and since $w \in W^s_{\loc}(y)\cap 
	W^u_{\loc}(z)$, there exists $C_3 > 0$ such that 
	$$d_{\Omega}(y,w)^{\alpha\beta} + d_{\Omega}(w,z)^{\alpha\beta} \leq 
	C_3\, d_{\Omega}(y,z)^{\alpha\beta},$$ which concludes the proof.
	
\end{proof}

Using that the extension of a H\"older map to the closure of the homoclinic class $W^s(x_0)\cap W^u(x_0)$ is also a H\"older map, we have a conjugacy defined in the whole space that is H\"older with respect to the uniform metric, and whose fibre maps belong to $\Hol\beta(M)$ with uniformly bounded H\"older constants for the maps and their inverses, by Lemmas \ref{lem:unif-fiber} and \ref{lem:limit-homeo}.

\subsection{Case when the shift has a periodic point}	
\label{sec:A-per}

Consider now $x_0 \in \Omega$ a periodic point for $\sigma$ with period 
$n_0$. Then $x_0$ is a fixed point for $\tilde{\sigma} = \sigma^{n_0}$ 
and we can apply the result proven in Section~\ref{sec:A-fixed}. Indeed, since $\tilde\sigma$ is topologically mixing, both cocycles $\tilde{f}_x = f^{n_0}_x$ and $\tilde{g}_x = 
g^{n_0}_x$ are Lipschitz and su-dominated over $\tilde{\sigma}$, and 
satisfy
$$\tilde{f}^{\operatorname{per}(y_0)}_{y_0} = 
\tilde{g}^{\operatorname{per}(y_0)}_{y_0}$$
for every periodic point $y_0$, it follows that there exists an $\alpha\beta$-H\"older 
map $\phi : \Omega \to (\mathcal{H}_{\beta}(M), \dsup)$ satisfying
$$\tilde{f}_{x} = \phi_{\tilde{\sigma}(x)}\, \tilde{g}_{x}\, 
\phi_{x}^{-1}$$
for all $x \in \Omega$, or equivalently
\begin{equation}  
	\label{eq:cocycle-per}  
	f^{n_0}_{x} = \phi_{\sigma^{n_0}(x)}\, g^{n_0}_{x}\, \phi_{x}^{-1}  
\end{equation}  
for all $x \in \Omega$. We are going to prove that this same $\phi$ 
satisfies the conclusion of Theorem~\ref{teo:A}. To achieve that, we 
apply the following lemma, with an analogous result holding for the 
unstable case.

\begin{lemma}
	\label{lem:hol-conj}
	For all $y \in W = W^s(x_0) \cap W^u(\sigma^{n_0-1}(x_0))$ and all $z \in W^s(y)$, we have
	$$\phi_z = h^{s,f}_{yz}\phi_y h^{s,g}_{zy}.$$
\end{lemma}

\begin{proof}
	Equation~\eqref{eq:cocycle-per} implies that for every $k \in 
	\mathbb{N}$ and $y \in \Omega$,
	\begin{equation}
		\label{eq:knzero}
		\phi_{\sigma^{kn_0}(y)} = f^{kn_0}_{y} \phi_{y} (g^{kn_0}_{y})^{-1}.
	\end{equation}
	Consider $y \in W$ and $z \in W^s(y)$ and define the sequence
	\begin{equation}
		\label{eq:lem-hol-conjI}
		a_k := (f^{kn_0}_z)^{-1} \phi_{\sigma^{kn_0}(y)} g^{kn_0}_z = 
		(f^{kn_0}_z)^{-1} f^{kn_0}_{y} \phi_{y} (g^{kn_0}_{y})^{-1} 
		g^{kn_0}_{z}.
	\end{equation}
	Proposition~\ref{prop:stab-manifold-holon} implies that $a_k$ 
	converges to $h^{s,f}_{yz}\phi_y h^{s,g}_{zy}$ in the uniform 
	metric as $k \to \infty$. Hence, since our goal is to prove that 
	$\phi_z = h^{s,f}_{yz} \phi_y h^{s,g}_{zy}$, it suffices to verify 
	that $a_k$ converges to $\phi_z$. From equation~\eqref{eq:knzero} 
	applied to $z$, we have $\phi_{z} = (f^{kn_0}_{z})^{-1} 
	\phi_{\sigma^{kn_0}(z)} g^{kn_0}_{z}$, and thus
	\begin{align*}
		\dsup(\phi_z, a_k) &= \dsup\!\left((f^{kn_0}_{z})^{-1} 
		\phi_{\sigma^{kn_0}(z)} g^{kn_0}_{z},\, (f^{kn_0}_z)^{-1} 
		\phi_{\sigma^{kn_0}(y)} g^{kn_0}_z\right) \\
		&= \dsup\!\left((f^{kn_0}_{z})^{-1} \phi_{\sigma^{kn_0}(z)},\, 
		(f^{kn_0}_z)^{-1} \phi_{\sigma^{kn_0}(y)}\right) \\
		&\leq L\!\left((f^{kn_0}_{z})^{-1}\right) 
		\dsup\!\left(\phi_{\sigma^{kn_0}(z)}, 
		\phi_{\sigma^{kn_0}(y)}\right).
	\end{align*}
	Since $f$ is su-dominated, $L((f^{kn_0}_{z})^{-1}) < 
	\rho^{kn_0(\alpha-\theta)}$ for some $\theta > 0$.

	To bound $\dsup(\phi_{\sigma^{kn_0}(z)}, \phi_{\sigma^{kn_0}(y)}) \leq \dsup(\phi_{\sigma^{kn_0}(y)}, \operatorname*{Id}) + \dsup(\operatorname*{Id}, \phi_{\sigma^{kn_0}(z)})$, we use the fact that, for $w \in W^s_{\loc}(x_0)$,
	\begin{align*}
		\dsup(\phi_{w}, \operatorname*{Id}) &= \dsup(h^{s,f}_{x_0w}h^{s,g}_{wx_0}, \operatorname*{Id})\\
		&\leq \dsup(h^{s,f}_{x_0w}, \operatorname*{Id}) + \dsup(h^{s,g}_{x_0w}, \operatorname*{Id})\\
		&\leq 2 C d_\Omega(x_0,w)^{\alpha},
	\end{align*}
	where the last inequality follows from Proposition \ref{prop:stab-manifold-holon}, valid since $w$ belongs to the local stable set of $x_0$. Since $y, z \in W^s(x_0)$, both $\sigma^{kn_0}(y)$ and $\sigma^{kn_0}(z)$ belong to $W^s_{\loc}(x_0)$ for every $k$ large enough, and the limit argument only concerns large $k$. Hence we have that 
	$$\dsup(\phi_{\sigma^{kn_0}(z)}, \phi_{\sigma^{kn_0}(y)}) \leq 4 C' \rho^{-\alpha kn_0}.$$
	
	Therefore
	$$\dsup(\phi_z, a_k) \leq \rho^{kn_0(\alpha-\theta)}\, 
	4 C' \rho^{-\alpha kn_0} = 4 C' \rho^{-kn_0\theta},$$
	which converges to $0$ as $k \to \infty$.
\end{proof}

To conclude the proof of Theorem~\ref{teo:A} in the mixing case, since $W = W^s(x_0) \cap 
W^u(\sigma^{n_0-1}(x_0))$ is dense, it suffices to prove that the 
cohomological equation holds for $y \in W$. Notice that $y \in 
W^u(\sigma^{n_0-1}(x_0))$ implies $\sigma(y) \in 
W^u(\sigma^{n_0}(x_0)) = W^u(x_0)$.

To prove that $\phi_y = f_y^{-1} \phi_{\sigma(y)} g_y$, we use 
Lemma~\ref{lem:hol-conj} to obtain
$$\phi_y = h^{s,f}_{x_0 y} \phi_{x_0} h^{s,g}_{y x_0} \quad 
\text{and} \quad \phi_{\sigma(y)} = h^{u,f}_{x_0\sigma(y)} \phi_{x_0} 
h^{u,g}_{\sigma(y)x_0}.$$  
Since $\phi_{x_0} = \operatorname{Id}$ by definition, these expressions 
simplify to  
\begin{align*}  
	\phi_y &= h^{s,f}_{x_0 y} h^{s,g}_{y x_0} = \lim_{n \to \infty} 
	(f^{n n_0}_y)^{-1} f^{n n_0}_{x_0} (g^{n n_0}_{x_0})^{-1} 
	g^{n n_0}_y, \\  
	\phi_{\sigma(y)} &= h^{u,f}_{x_0\sigma(y)} h^{u,g}_{\sigma(y)x_0} 
	= \lim_{n \to \infty} (f^{-n n_0}_{\sigma(y)})^{-1} f^{-n n_0}_{x_0} 
	(g^{-n n_0}_{x_0})^{-1} g^{-n n_0}_{\sigma(y)}.  
\end{align*}  
It is therefore sufficient to show that
\begin{align*}
	\lim_{n \to \infty} (f^{n n_0}_y)^{-1} g^{n n_0}_y 
	&= f_y^{-1} \left( \lim_{n \to \infty} (f^{-n n_0}_{\sigma(y)})^{-1} 
	g^{-n n_0}_{\sigma(y)} \right) g_y \\
	&= \lim_{n \to \infty} (f^{-n n_0+1}_y)^{-1} g^{-n n_0+1}_y.
\end{align*}
Fix $\varepsilon > 0$. Since $y \in W$, there exists $N \in \mathbb{N}$ 
such that for every $n \geq N$,
$$d\!\left(\sigma^{-n n_0 + 1}(y), \sigma^{n n_0}(y)\right) < \varepsilon.$$
As in the proof of Lemma~\ref{lema1}, let $z_n \in \Omega$ be a 
$\sigma$-periodic point shadowing the $\varepsilon$-pseudo orbit 
$\{\sigma^{-n n_0 + 1}(y), \ldots, \sigma^{n n_0}(y)\}$. Using that
$$(f^{-n n_0 +1}_{z_n})^{-1} g^{-n n_0 +1}_{z_n} = 
(f^{n n_0}_{z_n})^{-1} g^{n n_0}_{z_n},$$
since $z_n$ is $(2n n_0 -1)$-periodic and 
\begin{align*}
	g^{n n_0}_{z_n} (g^{-n n_0 +1}_{z_n})^{-1} &= g^{n n_0}_{z_n} 
	g^{n n_0-1}_{\sigma^{-n n_0 +1}(z_n)} = g^{2n n_0 -1}_{\sigma^{-n n_0 +1}(z_n)} \\
	&= f^{2n n_0 -1}_{\sigma^{-n n_0 +1}(z_n)} = f^{n n_0}_{z_n} 
	(f^{-n n_0 +1}_{z_n})^{-1},
\end{align*}
we bound $\dsup\!\left((f^{n n_0}_y)^{-1} g^{n n_0}_y,\, 
(f^{-n n_0+1}_y)^{-1} g^{-n n_0+1}_y\right)$ with
\begin{equation}
	\label{eq:per-lim}
	\begin{aligned}
		\dsup\!((f^{n n_0}_y)^{-1} g^{n n_0}_y,\, 
		&(f^{n n_0}_{z_n})^{-1} g^{n n_0}_{z_n})\\
		+\, &\dsup\!\left((f^{-n n_0+1}_{z_n})^{-1} g^{-n n_0+1}_{z_n},\, 
		(f^{-n n_0+1}_y)^{-1} g^{-n n_0+1}_y\right),
	\end{aligned}
\end{equation}
where both terms converge to zero as $n$ increases. This concludes 
the proof of Theorem~\ref{teo:A} when $\sigma$ is topologically mixing.

\subsection{The transitive case}
\label{sec:A-transitive}

Suppose now that $\mathcal{A}$ is irreducible with period $p > 1$, and let $\Omega = \Omega_0 \sqcup \cdots \sqcup \Omega_{p-1}$ be the cyclic decomposition, with $\sigma(\Omega_j) = \Omega_{j+1 \bmod p}$ and $\sigma^p$ topologically mixing on each $\Omega_j$.

Consider the cocycles $f^p$ and $g^p$ over $\sigma^p|_{\Omega_0}$. They satisfy the hypotheses of the mixing case, since both are Lipschitz and su-dominated over $\sigma^p$, and 
satisfy
$$(f^p)^m_x = f^{pm}_x = g^{pm}_x = (g^p)^m_x,$$
so $f^p$ and $g^p$ have coincidence of periodic data over $\sigma^p|_{\Omega_0}$. By Sections \ref{sec:A-fixed} and \ref{sec:A-per}, there are $\beta \in (0,1)$ and an $\alpha\beta$-H\"older map $\phi^0 : \Omega_0 \to (\Hol\beta(M), \dsup)$ such that
\begin{equation}
	\label{eq:power-conj}
	f^p_x = \phi^0_{\sigma^p(x)}\, g^p_x\, (\phi^0_x)^{-1} \quad \text{for all } x \in \Omega_0.
\end{equation}

Define $\phi : \Omega \to \Hol\beta(M)$ by
$$\phi_{\sigma^j(x)} := f^j_x\, \phi^0_x\, (g^j_x)^{-1}, \qquad x \in \Omega_0,\ 0 \leq j \leq p-1.$$
This is well defined because $\sigma^j : \Omega_0 \to \Omega_j$ is a homeomorphism. We verify the cohomological equation. For $y = \sigma^j(x) \in \Omega_j$ with $0 \leq j < p - 1$,
\begin{align*}
	\phi_{\sigma(y)}\, g_y\, \phi_y^{-1} &= f^{j+1}_x\, \phi^0_x\, (g^{j+1}_x)^{-1} g_{\sigma^j(x)}\, g^j_x\, (\phi^0_x)^{-1} (f^j_x)^{-1}\\
	&= f^{j+1}_x (f^j_x)^{-1} = f_{\sigma^j(x)} = f_y.
\end{align*}
For $j = p - 1$, we have $\sigma(y) = \sigma^p(x) \in \Omega_0$ and $\phi_{\sigma(y)} = \phi^0_{\sigma^p(x)}$, and the same computation, together with equation \eqref{eq:power-conj}, gives
$$\phi_{\sigma(y)}\, g_y\, \phi_y^{-1}
= \phi^0_{\sigma^p(x)}\, g^{p}_x\, (\phi^0_x)^{-1} (f^{p-1}_x)^{-1}
= f^p_x\, (f^{p-1}_x)^{-1} = f_y.$$

It remains to check the regularity of $\phi$. On each $\Omega_j$, the map $\phi$ is obtained from $\phi^0$ by composing with the fibre maps $f^j_x$ and $(g^j_x)^{-1}$, which are Lipschitz with constants bounded by $L_0^j$ with $L_0 = \sup_z \max\{L(f_z), L(g_z^{-1})\}$, and with the base map $\sigma^{-j} : \Omega_j \to \Omega_0$, which is bi-Lipschitz for $d_\Omega$. A direct computation using Lemma \ref{lem:d1} and the $\alpha$-H\"older regularity of $x \mapsto f^j_x$ and $x \mapsto (g^j_x)^{-1}$ then shows that $\phi$ is $\alpha\beta$-H\"older with respect to $\dsup$ on each $\Omega_j$, hence on $\Omega$, since the sets $\Omega_j$ are closed and pairwise disjoint, and that each $\phi_y$ belongs to $\Hol\beta(M)$ with uniformly bounded H\"older constants. This proves Theorem \ref{teo:A} for every transitive subshift of finite type.

\section{Proof of Theorem \ref{teo:B}}
\label{sec:proofB}

Unlike the proof of Theorem \ref{teo:A}, the argument below does not use homoclinic classes, so no mixing assumption is needed. Notice also that the existence of an ergodic measure with full support already implies that $\sigma$ is transitive, since $\mu$-almost every point has a dense orbit in the support of $\mu$.

Let $\phi : \Omega \to \mathcal{H}_{\beta}(M)$ be a $\mu$-measurable conjugacy between $f$ 
and $g$, that is, $\phi$ is measurable with respect to the Borel 
$\sigma$-algebra on $(\mathcal{H}_{\beta}(M), \dsup)$ and satisfies
$$f_{x} = \phi_{\sigma(x)}\, g_{x}\, \phi^{-1}_{x}$$
for $\mu$-almost every $x \in \Omega$. We first show that, along stable 
and unstable manifolds, the holonomies satisfy a conjugacy relation with 
$\phi$ on a set of full $\mu$-measure.

\begin{lemma}
    \label{lem:conj-hol}
    There is a set $\Gamma \subseteq \Omega$ with $\mu(\Gamma) = 1$ such that
    $$h^{f,s}_{xy} = \phi_y h^{g,s}_{xy} \phi_x^{-1}$$
    for $x, y \in \Gamma$ such that $y \in W^s(x)$. The same holds for unstable holonomies.
\end{lemma}

\begin{proof}
    Taking $x, y \in \Omega$ such that $y \in W^s(x)$, a simple computation shows that
    \begin{equation*}
        \label{eq:lemB1}
        (f_y^n)^{-1}f_x^n = \phi_{y} (g_{y}^n)^{-1} \phi_{\sigma^n(y)}^{-1} \phi_{\sigma^n(x)} g_{x}^n \phi^{-1}_{x},        
    \end{equation*}
   Using the definition of $h^{f,s}_{xy}$, we have that
    \begin{equation*}    
        h^{f,s}_{xy}=\lim_{n\to\infty} \phi_{y} (g_{y}^n)^{-1} \phi_{\sigma^n(y)}^{-1} \phi_{\sigma^n(x)} g_{x}^n \phi^{-1}_{x},
    \end{equation*}
    therefore it is enough to verify that $\lim \phi_{\sigma^n(y)}^{-1} \phi_{\sigma^n(x)} = \operatorname{Id}$ for $\mu$-almost every pair $x, y$. 
    
    Indeed, if we define two sequences $a_n, b_n \in \operatorname{Homeo}(M)$ by
    \begin{align*}
        &a_n = \phi_{y} (g_{y}^n)^{-1} \phi_{\sigma^n(y)}^{-1} \phi_{\sigma^n(x)} g_{x}^n \phi^{-1}_{x}\\
        &b_n = \phi_{y} (g_{y}^n)^{-1} g_{x}^n \phi^{-1}_{x},
    \end{align*}
    we want to prove that $\lim \dsup(a_n, b_n) = \dsup(h^{f,s}_{xy}, \phi_y h^{g,s}_{xy} \phi_x^{-1}) = 0$. For every $y \in \Omega$ we have that
    \begin{align*}
        \dsup(a_n, b_n) &= d_{\infty}(\phi_{y} (g_{y}^n)^{-1} \phi_{\sigma^n(y)}^{-1} \phi_{\sigma^n(x)}, \phi_{y} (g_{y}^n)^{-1})\\
        &\leq H_{\beta}(\phi_{y}) d_{\infty}((g_{y}^n)^{-1} \phi_{\sigma^n(y)}^{-1} \phi_{\sigma^n(x)},  (g_{y}^n)^{-1})^\beta\\
        &\leq H_{\beta}(\phi_{y}) K_g^\beta \dsup(\phi_{\sigma^n(y)}^{-1} \phi_{\sigma^n(x)}, Id)^\beta,
    \end{align*}
    where $K_g$ is the constant of bounded distortion in Definition \ref{def:bounded-dist}. Then if $\lim \phi_{\sigma^n(y)}^{-1} \phi_{\sigma^n(x)} = \operatorname{Id}$ for $\mu$-almost every pair $x, y$, the result follows.

    To prove that $\dsup(\phi_{\sigma^n(y)}^{-1} \phi_{\sigma^n(x)}, Id)$ converges to zero, consider that
    \begin{align*}
        \dsup(\phi_{\sigma^n(y)}^{-1} \phi_{\sigma^n(x)}, \Id) = \dsup(\phi_{\sigma^n(y)}^{-1} , \phi_{\sigma^n(x)}^{-1}).
    \end{align*}
As $y \in W^s(x)$, the sequence $d_\Omega(\sigma^n(x), \sigma^n(y))$ converges to $0$.

Since $\phi$ is measurable, Lusin's Theorem implies that for any $\varepsilon > 0$ with $\varepsilon < 1/2$, there exists a compact set $ \mathcal{K} \subseteq \Omega $ with $\mu(\mathcal{K}) > 1 - \varepsilon$ such that $\phi$ and $\phi^{-1}$ are uniformly continuous on $\mathcal{K}$. By the Birkhoff ergodic theorem, the set $\Gamma$ of points visiting $\mathcal{K}$ with frequency at least $1 - \varepsilon$ has full measure. If $x, y \in \Gamma$, the set of common return times $\{n : \sigma^n(x) \in \mathcal{K} \text{ and } \sigma^n(y) \in \mathcal{K}\}$ has lower density at least $1 - 2\varepsilon > 0$, so it is infinite. Along this subsequence $(n_k)$, since $\phi^{-1}$ is uniformly continuous on $\mathcal{K}$ and $d_\Omega(\sigma^{n_k}(x), \sigma^{n_k}(y)) \to 0$, we obtain $d_\infty(\phi^{-1}_{\sigma^{n_k}(x)}, \phi^{-1}_{\sigma^{n_k}(y)}) \to 0$.

It remains to observe that convergence along a subsequence is enough to conclude. Both sequences $a_n$ and $b_n$ converge in $\dsup$ along the full sequence: $a_n = (f^n_y)^{-1} f^n_x \to h^{f,s}_{xy}$ by Proposition \ref{prop:stab-manifold-holon}, and $b_n \to \phi_y\, h^{g,s}_{xy}\, \phi_x^{-1}$, since $(g^n_y)^{-1} g^n_x \to h^{g,s}_{xy}$ in $\dsup$, composition on the right is an isometry for $\dsup$ by item (1) of Lemma \ref{lem:d1}, and composition on the left with the fixed map $\phi_y$ is $\dsup$-continuous, as $\dsup(\phi_y \circ u, \phi_y \circ v) \leq H_\beta(\phi_y)\, \dsup(u, v)^\beta$. Since $\dsup(a_{n_k}, b_{n_k}) \to 0$ along the subsequence, the two limits coincide, that is, $h^{f,s}_{xy} = \phi_y\, h^{g,s}_{xy}\, \phi_x^{-1}$.
\end{proof}

We now prove that there exists a constant $C_1 > 0$ such that
$$\dsup(\phi_x, \phi_y) < C_1\, d_{\Omega}(x, y)^{\alpha\beta}$$
holds for $x, y$ in a dense set.

We begin by using Lemma~\ref{lem:conj-hol} to 
establish this inequality for $x, y \in \Gamma$ with $y \in W^s_{\loc}(x)$. 
In this case $\phi_y = h^{f,s}_{xy} \phi_x h^{g,s}_{yx}$, and
\begin{align*}
	\dsup(\phi_x, \phi_y) &= \dsup\!\left(h^{f,s}_{xy} 
	h^{f,s}_{yx} \phi_x,\, h^{f,s}_{xy} \phi_x h^{g,s}_{yx}\right)\\
	&\leq L(h^{f,s}_{xy})\, \dsup\!\left(h^{f,s}_{yx} 
	\phi_x,\, \phi_x h^{g,s}_{yx}\right)\\
	&\leq L(h^{f,s}_{xy}) \left( \dsup(h^{f,s}_{yx}
	\phi_x, \phi_x) + \dsup(\phi_x, \phi_x 
	h^{g,s}_{yx}) \right)\\
	&\leq L(h^{f,s}_{xy})\,(1 + H_{\beta}(\phi_x))\, 
	\max\!\left\{\dsup(h^{f,s}_{yx}, \operatorname{Id}),\, 
	\dsup(h^{g,s}_{yx}, \operatorname{Id})^{\beta}\right\}\\
	&\leq L(h^{f,s}_{xy})\,(1 + H_{\beta}(\phi_x))\, 
	C_2\, d_{\Omega}(x, y)^{\alpha\beta},
\end{align*}
where $C_2 > 0$ comes from Proposition \ref{prop:stab-manifold-holon}. The bounded distortion implies that the holonomies of $f$ are Lipschitz with a Lipschitz constant equal to $K_f^2$. By hypothesis, $H_{\beta}(\phi_x)$ is uniformly bounded. So the above inequality becomes 
$$\dsup(\phi_x, \phi_y) \leq C_3 d_{\Omega}(x, y)^{\alpha\beta},$$
for $x, y \in \Gamma$ with $y \in W^s_{\loc}(x)$, where 
$C_3 = \sup_{x \in \Gamma} K_f^2(1 + H_{\beta}
(\phi_x))\, C_2 < \infty$. 
The same argument applies to $x, y \in 
\Gamma$ in the same local unstable set.

Now consider an open rectangle $R \subset \Omega$, that is, an open set that is closed with 
respect to the bracket operation $[x,y] = W^s_{\loc}(x) \cap W^u_{\loc}(y)$. Since 
$\mu$ has local product structure,
$$\restr{\mu}{R} \ll \mu_{x_0}^u \times \mu_{x_0}^s$$
for some $x_0 \in R$. Using this together with the full support of $\mu$, 
we show that for $\mu$-almost every $x, y \in R \cap \Gamma$,
\begin{equation}
	\label{eq:lip-full-meas}
	\dsup(\phi_x, \phi_y) \leq C_4\, d_{\Omega}(x, y)^{\alpha\beta}.
\end{equation}
Indeed, for every such $R$, almost every local stable leaf $W^s_{\loc}(x_0)$ 
in $R$ intersects $\Gamma$ in a set of full $\mu_{x_0}^s$-measure, and 
$\mu_{x_0}^s$ has full support since $\mu$ does. Thus, given $x, y \in R$ 
whose local stable leaves both intersect $\Gamma$ with full conditional 
measure, there exists $z \in W^s_{\loc}(x)$ such that 
$w = W^u_{\loc}(z) \cap W^s_{\loc}(y)$ belongs 
to $R \cap \Gamma$. Applying the H\"older inequality above to the pairs 
$(x,z)$, $(z,w)$, and $(w,y)$ with constant $C_3$, and using that there 
exists $C_5 > 0$ with
$$d_\Omega(x,z)^{\alpha\beta} + d_\Omega(z,w)^{\alpha\beta} + 
d_\Omega(w,y)^{\alpha\beta} \leq C_5\, d_\Omega(x,y)^{\alpha\beta},$$
inequality~\eqref{eq:lip-full-meas} follows.

Finally, consider
$$\tilde\Gamma = \bigcap_{n \in \mathbb{N}} \sigma^n(\Gamma_0),$$
where $\Gamma_0 \subseteq \Gamma$ is the full measure set satisfying 
inequality~\eqref{eq:lip-full-meas}. This is a $\sigma$-invariant dense 
set on which the cohomological equation
$$f_{x} = \phi_{\sigma(x)}\, g_{x}\, \phi^{-1}_{x}$$
holds for every $x \in \tilde\Gamma$, with $\phi$ being H\"older with 
respect to the uniform metric. Hence $\phi$ coincides $\mu$-almost 
everywhere with a H\"older continuous map $\tilde\phi : \Omega \to 
(\mathcal{H}_\beta(M), \dsup)$, since $\phi$ and $\tilde{\phi}$ coincide on a dense set. Moreover, by hypothesis, $H_\beta(\phi_x)$ and $H_\beta(\phi_x^{-1})$ are uniformly bounded on $\tilde\Gamma$, so Lemma \ref{lem:limit-homeo} guarantees that $\tilde\phi_x \in \mathcal{H}_\beta(M)$ for every $x \in \Omega$.

\bibliographystyle{plainurl}
\bibliography{references} 	

\begin{thebibliography}{10}

\bibitem{AKL18}
Artur Avila, Alejandro Kocsard, and Xiao-Chuan Liu.
\newblock Liv\v{s}ic theorem for diffeomorphism cocycles.
\newblock {\em Geom. Funct. Anal.}, 28(4):943--964, 2018.
\newblock \href {https://doi.org/10.1007/s00039-018-0454-y} {\path{doi:10.1007/s00039-018-0454-y}}.

\bibitem{AV2010}
Artur Avila and Marcelo Viana.
\newblock Extremal {L}yapunov exponents: an invariance principle and applications.
\newblock {\em Invent. Math.}, 181(1):115--189, 2010.
\newblock \href {https://doi.org/10.1007/s00222-010-0243-1} {\path{doi:10.1007/s00222-010-0243-1}}.

\bibitem{Backes2015}
Lucas Backes.
\newblock Rigidity of fiber bunched cocycles.
\newblock {\em Bull. Braz. Math. Soc. (N.S.)}, 46(2):163--179, 2015.
\newblock \href {https://doi.org/10.1007/s00574-015-0089-7} {\path{doi:10.1007/s00574-015-0089-7}}.

\bibitem{BackesKocsard}
Lucas Backes and Alejandro Kocsard.
\newblock Cohomology of dominated diffeomorphism-valued cocycles over hyperbolic systems.
\newblock {\em Ergodic Theory Dynam. Systems}, 36(6):1703--1722, 2016.
\newblock \href {https://doi.org/10.1017/etds.2014.149} {\path{doi:10.1017/etds.2014.149}}.

\bibitem{BP19}
Lucas Backes and Mauricio Poletti.
\newblock A {L}iv\v{s}ic theorem for matrix cocycles over non-uniformly hyperbolic systems.
\newblock {\em J. Dynam. Differential Equations}, 31(4):1825--1838, 2019.
\newblock \href {https://doi.org/10.1007/s10884-018-9691-x} {\path{doi:10.1007/s10884-018-9691-x}}.

\bibitem{BN98}
Hari Bercovici and Viorel Ni\c{t}ic\u{a}.
\newblock A {B}anach algebra version of the {L}ivsic theorem.
\newblock {\em Discrete Contin. Dynam. Systems}, 4(3):523--534, 1998.
\newblock \href {https://doi.org/10.3934/dcds.1998.4.523} {\path{doi:10.3934/dcds.1998.4.523}}.

\bibitem{bowen70}
Rufus Bowen.
\newblock Markov partitions for {A}xiom {${\rm A}$} diffeomorphisms.
\newblock {\em Amer. J. Math.}, 92:725--747, 1970.
\newblock \href {https://doi.org/10.2307/2373370} {\path{doi:10.2307/2373370}}.

\bibitem{Butler2018}
Clark Butler.
\newblock Measurable rigidity of the cohomological equation for linear cocycles over hyperbolic systems.
\newblock {\em Israel J. Math.}, 227(1):27--61, 2018.
\newblock \href {https://doi.org/10.1007/s11856-018-1733-x} {\path{doi:10.1007/s11856-018-1733-x}}.

\bibitem{ddLW2010}
Rafael de~la Llave and Alistair Windsor.
\newblock Liv\v{s}ic theorems for non-commutative groups including diffeomorphism groups and results on the existence of conformal structures for {A}nosov systems.
\newblock {\em Ergodic Theory Dynam. Systems}, 30(4):1055--1100, 2010.
\newblock \href {https://doi.org/10.1017/S014338570900039X} {\path{doi:10.1017/S014338570900039X}}.

\bibitem{HPS}
Morris Hirsch, Charles Pugh, and Michael Shub.
\newblock {\em Invariant manifolds}.
\newblock Lecture Notes in Mathematics, Vol. 583. Springer-Verlag, Berlin-New York, 1977.

\bibitem{Kalinin2011}
Boris Kalinin.
\newblock Liv\v{s}ic theorem for matrix cocycles.
\newblock {\em Ann. of Math. (2)}, 173(2):1025--1042, 2011.
\newblock \href {https://doi.org/10.4007/annals.2011.173.2.11} {\path{doi:10.4007/annals.2011.173.2.11}}.

\bibitem{Kalinin-Sadovskaya}
Boris Kalinin and Victoria Sadovskaya.
\newblock Linear cocycles over hyperbolic systems and criteria of conformality.
\newblock {\em J. Mod. Dyn.}, 4(3):419--441, 2010.
\newblock \href {https://doi.org/10.3934/jmd.2010.4.419} {\path{doi:10.3934/jmd.2010.4.419}}.

\bibitem{KP16}
Alejandro Kocsard and Rafael Potrie.
\newblock Liv\v{s}ic theorem for low-dimensional diffeomorphism cocycles.
\newblock {\em Comment. Math. Helv.}, 91(1):39--64, 2016.
\newblock \href {https://doi.org/10.4171/CMH/377} {\path{doi:10.4171/CMH/377}}.

\bibitem{LindMarcus2021}
Douglas Lind and Brian Marcus.
\newblock {\em An introduction to symbolic dynamics and coding}.
\newblock Cambridge Mathematical Library. Cambridge University Press, Cambridge, 2021.
\newblock Second edition [of 1369092].
\newblock \href {https://doi.org/10.1017/9781108899727} {\path{doi:10.1017/9781108899727}}.

\bibitem{Liv71}
Alexander~N. Liv\v{s}ic.
\newblock Certain properties of the homology of {$Y$}-systems.
\newblock {\em Mat. Zametki}, 10:555--564, 1971.

\bibitem{Liv72}
Alexander~N. Liv\v{s}ic.
\newblock Cohomology of dynamical systems.
\newblock {\em Izv. Akad. Nauk SSSR Ser. Mat.}, 36:1296--1320, 1972.

\bibitem{NT95}
Viorel Ni\c{t}ic\u{a} and Andrei T\"{o}r\"{o}k.
\newblock Cohomology of dynamical systems and rigidity of partially hyperbolic actions of higher-rank lattices.
\newblock {\em Duke Math. J.}, 79(3):751--810, 1995.
\newblock \href {https://doi.org/10.1215/S0012-7094-95-07920-4} {\path{doi:10.1215/S0012-7094-95-07920-4}}.

\bibitem{PollicottWalkden2001}
Mark Pollicott and Charles Walkden.
\newblock Liv\v{s}ic theorems for connected {L}ie groups.
\newblock {\em Trans. Amer. Math. Soc.}, 353(7):2879--2895, 2001.
\newblock \href {https://doi.org/10.1090/S0002-9947-01-02708-8} {\path{doi:10.1090/S0002-9947-01-02708-8}}.

\bibitem{Sadovskaya}
Victoria Sadovskaya.
\newblock Cohomology of fiber bunched cocycles over hyperbolic systems.
\newblock {\em Ergodic Theory Dynam. Systems}, 35(8):2669--2688, 2015.
\newblock \href {https://doi.org/10.1017/etds.2014.43} {\path{doi:10.1017/etds.2014.43}}.

\bibitem{S99}
Klaus Schmidt.
\newblock Remarks on {L}iv\v{s}ic' theory for nonabelian cocycles.
\newblock {\em Ergodic Theory Dynam. Systems}, 19(3):703--721, 1999.
\newblock \href {https://doi.org/10.1017/S0143385799146790} {\path{doi:10.1017/S0143385799146790}}.

\bibitem{Wilk13}
Amie Wilkinson.
\newblock The cohomological equation for partially hyperbolic diffeomorphisms.
\newblock {\em Ast\'{e}risque}, (358):75--165, 2013.

\bibitem{ZC2021}
Rui Zou and Yongluo Cao.
\newblock Liv\v{s}ic theorems for {B}anach cocycles: existence and regularity.
\newblock {\em J. Funct. Anal.}, 280(5):Paper No. 108889, 37, 2021.
\newblock \href {https://doi.org/10.1016/j.jfa.2020.108889} {\path{doi:10.1016/j.jfa.2020.108889}}.

\end{thebibliography}
\end{document}